\documentclass[12pt,a4paper]{amsart}
\usepackage{amsfonts}
\usepackage{amsthm}
\usepackage{amsmath}
\usepackage{amscd}
\usepackage[latin2]{inputenc}
\usepackage{t1enc}
\usepackage[mathscr]{eucal}
\usepackage{indentfirst}
\usepackage{graphicx}
\usepackage{graphics}
\usepackage{pict2e}
\usepackage{epic}
\numberwithin{equation}{section}
\usepackage[margin=2.9cm]{geometry}
\usepackage[all]{xy}
\usepackage{amssymb}
\usepackage{mathtools}
\usepackage[toc]{appendix}
\usepackage{MnSymbol}
\usepackage{stackrel}
\usepackage[new]{old-arrows}
\usepackage[dvipsnames]{xcolor}
\usepackage{tikz}
\usepackage{tikz-cd}
\usepackage[backend=biber, style=alphabetic, sorting=ynt
]{biblatex}

\usepackage[dvipsnames]{xcolor} 
    \definecolor{red_cite}{RGB}{75, 128, 82}

\usepackage[
            colorlinks=true,        
            allcolors = black,  
            citecolor=red_cite,        
        ]{hyperref}

\usepackage[hyphenbreaks]{breakurl}
 
\hypersetup{
    colorlinks=true,
    linkcolor=blue,
    filecolor=magenta,      
    urlcolor=blue,
    breaklinks=true,
} 

\theoremstyle{plain}
\newtheorem{Th}{Theorem}[section]
\newtheorem{Lemma}[Th]{Lemma}
\newtheorem{Cor}[Th]{Corollary}
\newtheorem{prop}[Th]{Proposition}

\theoremstyle{definition}
\newtheorem{Def}[Th]{Definition}

\newtheorem{Rem}[Th]{Remark}
\newtheorem{?}[Th]{Problem}

\newcommand{\Mod}[1]{#1\text{-}\textrm{Mod}}

\title{Describing model categories througth homotopy tiny objects}
\author{Anna Giulia Montaruli}

\begin{document}

\begin{abstract}
Let $\mathcal C$ be a $\mathcal V$-enriched model category. We say that an object $x$ of $\mathcal C$ is homotopy tiny if the total right derived functor of $\mathcal C(x, -) : \mathcal{C} \rightarrow {\mathcal V}$ preserves homotopy weighted colimits. Let $\mathcal C_0$ be a full subcategory of $\mathcal C$ all of whose objects are homotopy tiny. Our main result says that the homotopy category of the category generated by $\mathcal C_0$ under weak equivalences and homotopy weighted colimits is equivalent to the homotopy category of the category $\mathcal V^{\mathcal C_0^{op}}$ of $\mathcal V$-enriched presheaves on $\mathcal C_0$ with values in $\mathcal V$. If $\mathcal C$ is generated by $\mathcal C_0$, then $\mathcal C$ is Quillen equivalent to $\mathcal V^{\mathcal C_0^{op}}$. Two special cases of our theorem are Schwede--Shipley's theorem on stable model categories and Elmendorf's theorem on equivariant spaces.
\end{abstract}

\maketitle

\section{Introduction}
An object $x$ of a category $\mathcal{C}$ is tiny if the functor $\mathcal{C}(x, -): \mathcal{C} \rightarrow \mathrm{Sets}$ preserves all small colimits. Common examples of tiny objects are the singleton in the category $\mathrm{Sets}$, or the retracts of representable functors in presheaf categories (see \cite[Proposition 2]{CauchyComplection}). On the other hand, for every ring $R$, the category $\Mod{R}$ of modules over $R$ has no tiny objects.

If the category $\mathcal{C}$ is enriched over a symmetric monoidal category $\mathcal{V}$, we could also talk about tiny objects in an enriched sense: these are the ones for which the $\mathcal{V}$-enriched functor $\mathcal{C}(x , -) : \mathcal{C} \rightarrow \mathcal{V}$ preserves all small colimits. For instance, it is well known that, in an enriched sense, the tiny objects of $\Mod{R}$ are the finitely generated projective modules (see \cite[Example 2]{CauchyComplection}).

If the category $\mathcal{C}$ is generated by tiny objects under colimits, then the Yoneda embedding induces an equivalence between $\mathcal{C}$ and the category of presheaves over the full subcategory of $\mathcal{C}$ of tiny objects; the same holds in an enriched setting too (see \cite[Theorem 5.26]{Kelly}). The main goal of the paper is to show a homotopical version of this result, i.e.\ to employ homotopy tiny objects (objects $x$ for which the total right derived functor of $\mathcal{C}(x, -): \mathcal{C} \rightarrow \mathcal{V}$ preserves small homotopy weighted colimits) to show a Quillen equivalence between a model category $\mathcal{C}$, generated by its homotopy tiny objects under homotopy weighted colimits and weak equivalences, and the category of $\mathcal{V}$-enriched presheaves over the full subcategory of homotopy tiny objects of $\mathcal{C}$.
This will allow us to unify two well known results, namely Elmendorf's Theorem, which states, under certain conditions, that the category of $G$-spaces is Quillen equivalent to the presheaf category over the category of orbits, and Schwede--Shipley's Theorem (see \cite[Theorem 3.9.3]{Schwede-Shipley}), which asserts that, in a spectral model category, the full subcategory generated by a set of compact generators is Quillen equivalent to the spectral presheaf category $(Sp^{\Sigma})^{\mathcal{G}^{op}}$. 

The outline of the paper is the following. In Section \ref{notation} we list the notation used. In Section \ref{preliminaries} we give some preliminary definitions and we recall some already known results. Section \ref{results} introduces homotopy tiny objects, gives some intermediate statements and ends with the main results (Theorem \ref{Th} and Corollary \ref{result}). Section \ref{examples} shows how Elmendorf's Theorem and Schwede--Shipley's Theorem can be seen as consequences of Corollary \ref{result}.

In the whole exposition, we assume the reader has some familiarity with the notion of model category and some other related topics, such as (total) left and (total) right derived functors, Quillen adjunctions and Quillen equivalences. We refer to \cite{Hovey} for a detailed explanation. The full understanding of the paper also requires some basic knowledge in enriched category theory: the main reference for this is \cite{Kelly}.

\subsection*{Acknowledgements}
The proposal of this work comes from my PhD co-advisor Gregory Arone, to whom I'm very grateful for his supervision in the realization of the paper. I would also like to thanks Peter LeFanu Lumsdaine for some valuable comments and suggestions on a first draft of the paper.

\section{Notation and assumptions}\label{notation}

Here below we list the notation used in the paper. 

\begin{itemize}
    \item[$\diamond$] by $x_c$ we denote the cofibrant replacement of the object $x$;
    \item[$\diamond$] by $y_f$ we denote the fibrant replacement of the object $y$;
    \item[$\diamond$] by $\gamma_{\mathcal{C}}$ we denote the canonical functor from a model category $\mathcal{C}$ to $Ho(\mathcal{C})$;
    \item[$\diamond$] by $x \sim y$ we mean that $x$ and $y$ are weakly equivalent, i.e.\ that there exists a zig zag of weak equivalences between them;
    \item[$\diamond$] by $\emptyset _\mathcal{C}$ we denote the initial object of a category $\mathcal{C}$; when it is clear from the context, we don't specify the category it belongs to;
    \item[$\diamond$] by \textit{llp} we denote the left lifting property;
    \item[$\diamond$] by $W \underset{\mathcal{D}}{\otimes} F$ we denote the colimit of the diagram $F : \mathcal{D} \rightarrow \mathcal{C}$ weighted by $W: \mathcal{D}^{op} \rightarrow \mathcal{V}$, or equivalently, the coend of $W$ and $F$ (see \cite[Proposition 4.1.5]{Loregian} for the equivalence);
    \item[$\diamond$] by $ x \xrightarrow{iso} y$ we mean that there is a canonical isomorphism between $x$ and $y$;
    \item[$\diamond$] by $\underset{\mathcal{D}}{\coprod} F$ we denote the coproduct of the small diagram $F$ over $\mathcal{D}$.
    \item[$\diamond$] by $y\underset{x}{\coprod} z$ we denote the pushout of the span $y \leftarrow x \rightarrow z$;
    \item[$\diamond$] by $y\times _x z$ we denote the pullaback of the span $y \rightarrow x \leftarrow z$.
    \item[$\diamond$] by $(-)^H : \mathcal{V}^G \rightarrow \mathcal{V}$ we deanote the $H$-fixed point functor, sending a $G$-space $X$ to $(X)^H \coloneqq \underset{H}{lim} X$ 
\end{itemize} 
Throughout the paper, we will assume that the fibrant and the cofibrant replacements are functorial assignments.

\section{Preliminaries}\label{preliminaries}

\subsection{First definitions}
\begin{Def}
A symmetric monoidal model category $\mathcal{V}$ is a symmetric monoidal category $(\mathcal{V}, \otimes, \mathbb{I})$ equipped with a model structure s.t.\ 
\begin{enumerate}
    \item (\textit{pushout-product axiom}) for any given cofibrations $x \xrightarrow{f} y$ and $z \xrightarrow{g} w$ in $\mathcal{V}$, the universal map induced by the pushout $(x \otimes w) \underset{x \otimes z}{\coprod} (y \otimes z) \rightarrow y \otimes w$ is cofibrant, and acyclic if $f$ or $g$ is so;
    \item the unit object $\mathbb{I}$ is cofibrant.
\end{enumerate}
\end{Def}
From now on, we will assume that $\mathcal{V}$ is a symmetric monoidal model category, whose model structure is also cofibrantly generated.
\begin{Def}\label{SP}
A $\mathcal{V}$-model category $\mathcal{C}$ is a $\mathcal{V}$-enriched category, tensored and cotensored over $\mathcal{V}$, equipped with a model structure such that, for every cofibration $f : a \rightarrow b$ and any fibration $g : c \rightarrow d$, the universal map $\mathcal{C}(b, c) \rightarrow \mathcal{C}(a, c) \times_{\mathcal{C}(a, d)} \mathcal{C}(b, d)$ is a fibration, acyclic if $f$ or $g$ are so.
\end{Def}

\subsection{The projective model structure of \texorpdfstring{$\mathcal{V}^{\mathcal{C}^{op}}$}{V-enriched presheaves over C}}
Recall that, given a model category $\mathcal{E}$ and a small category $\mathcal{D}$, the projective model structure, if it exists,  is the model structure on the functor category $\mathcal{E}^{\mathcal{D}}$ whose weak equivalences and fibrations are the maps that are objectwise weak equivalences and fibrations in $\mathcal{E}$. The existence of such a model structure is not always guaranteed, but it is in the special case we are going to consider, i.e.\ when $\mathcal{E}$ is the cofibrantly generated monoidal model category $\mathcal{V}$ (see \cite[Theorem 12.3.2]{Riehl}).

\begin{prop}\label{repr-functor-cofibrant}
Let $\mathcal{C}$ be a $\mathcal{V}$-model category. Then, every representable presheaf is cofibrant in the projective model structure of  $\mathcal{V}^{\mathcal{C}^{op}}$.
\end{prop}
\begin{proof}
Consider $y\in \mathcal{C}$, and take the representable presheaf $\mathcal{C}( - ,y)$. To show that this is a cofibrant object in the projective model structure of $\mathcal{V}^{\mathcal{C}^{op}}$, we need to show that the map $\emptyset \rightarrow \mathcal{C}(- ,y)$ has the \textit{llp} with respect to all the acyclic fibrations of $\mathcal{V}^{\mathcal{C}^{op}}$. Thanks to the canonical isomorphism 
\begin{equation*}
    \mathbb{I} \otimes \mathcal{C}( - ,y) \rightarrow \mathcal{C}( - ,y)
\end{equation*}
we can equivalently show that the map $\emptyset \rightarrow \mathbb{I} \otimes \mathcal{C}(- ,y)$
has the \textit{llp} with respect to all acyclic fibrations. Consider the diagram
\begin{equation*}
    \xymatrix{\emptyset \ar[r]\ar[d] & F \ar[d] \\
    \mathbb{I} \otimes \mathcal{C}( - , y) \ar[r] & G}
\end{equation*}
where $F\rightarrow G$ is an acyclic cofibration.
Using the cotensoredness of $\mathcal{C}$ over $\mathcal{V}$ and the Yoneda Lemma,
\begin{equation*}
    \mathcal{V}^{\mathcal{C}^{op}} (\mathbb{I} \otimes \mathcal{C}(- ,y) , F) \xrightarrow{iso} \mathcal{V} (\mathbb{I} , \mathcal{V}^{\mathcal{C}_0^{op}} ( \mathcal{C}( - , y) , F)) \xrightarrow{iso} \mathcal{C} (\mathbb{I} , Fy))
\end{equation*}
and
\begin{equation*}
    \mathcal{V}^{\mathcal{C}^{op}} (\mathbb{I} \otimes \mathcal{C}(- ,y) , G) \xrightarrow{iso} \mathcal{V} (\mathbb{I} , \mathcal{V}^{\mathcal{C}_0^{op}} ( \mathcal{C}( - , y) , G)) \xrightarrow{iso} \mathcal{C} (\mathbb{I} , Gy))
\end{equation*}
Since, in the projective model structure of $\mathcal{V}^{\mathcal{C}^{op}}$, acyclic fibrations are the one that are so componentwise in the model structure of $\mathcal{V}$, then the problem reduces to show the \textit{llp} for the map $\emptyset \rightarrow \mathbb{I}$ with respect to all acyclic fibrations $Fy \rightarrow Gy$ in $\mathcal{V}$; this follows immediately from the fact that $\mathbb{I}$ is cofibrant in the model structure of $\mathcal{V}$. 
\end{proof}

\subsection{Derived functors}

We first recall the definition of left (resp.\ right) derived and total left (resp.\ right) derived functors.

\begin{Def}
Let $\mathcal{C}$ be a model category, and $F : \mathcal{C} \rightarrow \mathcal{D}$ be a functor. The left derived functor of $F$ is a pair $(LF, t_{.})$, where $LF : Ho(\mathcal{C}) \rightarrow \mathcal{D}$ is a functor, and $t_{.} : LF \cdot \gamma_{\mathcal{C}} \rightarrow F$ is a natural transformation, which has the following universal property: for any other pair $(G : Ho(\mathcal{C}) \rightarrow \mathcal{D}, s_{.} : G\cdot \gamma_{\mathcal{C}} \rightarrow F)$, there exists a natural transformation $s'_{.} : G \rightarrow LF$ such that $t_{.} \cdot (s'_{\gamma_{\mathcal{C}} (-)}) = s_{.}$. For the sake of brevity, we simply denote by $LF$ the left derived functor of $F$.

The right derived functor of $F$ is defined dually, and denoted by $RF$.

\end{Def}

\begin{Def}
Let $\mathcal{C}$ be a model category. The total left derived functor $\mathbf{L}F$ of $F : \mathcal{C} \rightarrow \mathcal{D}$ is the left derived functor of the composition $\gamma_{\mathcal{D}}\cdot F : \mathcal{C} \rightarrow Ho(\mathcal{D})$. 

The total right derived functor of $F$ is defined dually, and denoted by $\mathbf{R}F$.
\end{Def}

The total left derived functor $(\mathbf{L} F)$ of a left Quillen adjoint 
$F : \mathcal{C} \rightarrow \mathcal{D}$ can be computed, at every $x\in\mathcal{C}$, as $(\mathbf{L} F)(\gamma_{\mathcal{C}} x) = \gamma_{\mathcal{D}} F(x_c)$. To show this, we recall two already well known results.

\begin{prop}\label{Quillen-adjoint} Let $F \dashv G$ a Quillen adjunction. Then
\begin{enumerate}
    \item $F$ preserves weak equivalences between cofibrant objects;
    \item $G$ preserves weak equivalences between fibrant objects.
\end{enumerate}
\end{prop}
\begin{proof}
See, for instance, \cite[Remark 9.8]{Dwyer-Spalinski}
\end{proof}

\begin{prop}\label{functor-preserving-weak-equivalences} Let $\mathcal{C}$ be a model category.
\begin{enumerate}
    \item Let $F : \mathcal{C} \rightarrow \mathcal{D}$ a functor with the property that $F(f)$ is an isomorphism whenever $f$ is a weak equivalence between cofibrant objects in $\mathcal{C}$. Then the left derived functor $(LF, t_{\cdot})$ of $F$ exists, and, for each cofibrant object $x$ of $\mathcal{C}$, the map $t_x : LF \cdot \gamma_{\mathcal{C}} (x) \rightarrow F(x)$ is an isomorphism;
    \item let $G : \mathcal{C} \rightarrow \mathcal{D}$ a functor with the property that $G(f)$ is an isomorphism whenever $f$ is a weak equivalence between fibrant objects in $\mathcal{C}$. Then the right derived functor $(RG, s_{\cdot})$ of $G$ exists, and, for each fibrant object $x$ of $\mathcal{C}$, the map $s_x : G(x) \rightarrow RG \cdot \gamma_{\mathcal{C}} (x)$ is an isomorphism.
\end{enumerate}
\end{prop}
\begin{proof}
See, for instance, \cite[Proposition 9.3]{Dwyer-Spalinski}
\end{proof}

We can now prove:
\begin{Lemma}\label{computing-left-derived}
Let $\mathcal{C}$ and $\mathcal{D}$ be model categories. 
\begin{enumerate}
    \item If $F : \mathcal{C} \rightarrow \mathcal{D}$ is a left Quillen adjoint functor, then, for every object $y\in\mathcal{C}$, $\mathbf{L}F (\gamma_{\mathcal{C}} y) \xrightarrow{iso} \gamma_{\mathcal{D}} F(y_c)$;
    \item If $G : \mathcal{C} \rightarrow \mathcal{D}$ is a right Quillen adjoint functor, then, for every object $y\in\mathcal{C}$, $\mathbf{R}G (\gamma_{\mathcal{C}} y) \xrightarrow{iso} \gamma_{\mathcal{D}} G(y_f)$.
\end{enumerate}
\end{Lemma}
\begin{proof}
\begin{enumerate}
    \item By Proposition \ref{Quillen-adjoint}, we know that $F$ preserves weak equivalences between cofibrant objects. Therefore, $\gamma_{\mathcal{D}}\cdot F : \mathcal{C} \rightarrow Ho(\mathcal{D})$ sends weak equivalences between cofibrant objects to isomorphisms (since $\gamma_{\mathcal{D}}$ sends weak equivalences to isomorphisms) and, by applying Proposition \ref{functor-preserving-weak-equivalences} to $ \gamma_{\mathcal{D}}\cdot F$, we obtain that $L( \gamma_{\mathcal{D}}\cdot F) = \mathbf{L}F$ exists.

Let $y \in \mathcal{C}$, and let $e: y_c \rightarrow y$ be a weak equivalence, and consider the diagram
\begin{equation*}
    \xymatrix{ \mathbf{L}F \cdot \gamma_{c} (y_c) \ar[r]^{\mathbf{L}F \cdot \gamma_{\mathcal{C}} (e)}\ar[d]_{t_{y_c}} & \mathbf{L}F \cdot \gamma_{\mathcal{C}} (y)\ar[d]^{t_y} \\
    \gamma_{\mathcal{D}}\cdot F(y_c) \ar[r]^{\gamma_{\mathcal{D}}\cdot F(e)} & \gamma_{\mathcal{D}}\cdot F(y)}
\end{equation*}
$t_{y_c}$ is an isomorphism by Proposition \ref{functor-preserving-weak-equivalences}, and $\mathbf{L}F \cdot \gamma_{\mathcal{C}} (e)$ is an isomorphism.
Therefore, $\gamma_{\mathcal{D}}\cdot F(y_c)$ and $\mathbf{L}F \cdot \gamma_{\mathcal{C}} (y)$ are isomorphic.
\item Using the dual argument to Point (1). \qedhere
\end{enumerate}
\end{proof}

From now on, we will make use of the isomorphism $(\gamma_{\mathcal{D}}\cdot F(e))^{-1} \cdot t_y$ to compute the object $\mathbf{L}F \cdot \gamma_{\mathcal{C}} (y)$; we will therefore write $\mathbf{L}F \cdot \gamma_{\mathcal{C}} (y) = \gamma_{\mathcal{D}}\cdot F(y_c)$. 

\subsection{Homotopy weighted colimits}

\begin{Def}
A functor $F$ between two model categories $\mathcal{C}$ and $\mathcal{D}$ is called homotopical if it preserves weak equivalences. When $\mathcal{C}$ and $\mathcal{D}$ are $\mathcal{V}$-model categories, we also require that a homotopical functor is $\mathcal{V}$-enriched.
\end{Def}

\begin{Rem}
Given a homotopical functor $F: \mathcal{C} \rightarrow \mathcal{D}$, there exists a functor $Ho(F) : Ho(\mathcal{C}) \rightarrow Ho(\mathcal{D})$ such that the following diagram commutes: 
\begin{equation*}
    \xymatrix{\mathcal{C}\ar[d]^{\gamma_{\mathcal{C}}} \ar[r]^F & \mathcal{D} \ar[d]^{\gamma_{\mathcal{D}}} \\
    Ho(\mathcal{C}) \ar[r]^{Ho(F)} & Ho(\mathcal{D})}
\end{equation*}
\end{Rem}

\begin{Def}
Let $\mathcal{C}$ be a $\mathcal{V}$-enriched category, and let $\mathcal{D}$ be a small category. We define the functor $- \underset{\mathcal{D}}{\otimes} - : \mathcal{V}^{\mathcal{D}^{op}} \times \mathcal{C}^{\mathcal{D}} \rightarrow \mathcal{C}$ as the functor sending the pair $(W, F)$ to $W \underset{\mathcal{D}}{\otimes} F$.
\end{Def}

\begin{Rem}
$- \underset{\mathcal{D}}{\otimes} -$ is a left Quillen bifunctor (see \cite[Definition A.3.1.1]{Lurie} for the definition of left Quillen bifunctor, and \cite[Remark A.2.9.27]{Lurie} for an explanation of why $- \underset{\mathcal{D}}{\otimes} -$ is a left Quillen bifunctor). This allows us to give the following definition.
\end{Rem}

\begin{Def}\label{hwcolim}
Given a $\mathcal{V}$-model category $\mathcal{C}$ and a small category $\mathcal{D}$, the homotopy weighted colimit functor, denoted by $- \underset{\mathcal{D}}{\otimes^h} -$, is defined to be the total left derived functor of the left Quillen bifunctor $- \underset{\mathcal{D}}{\otimes} - $.
\end{Def}

\begin{Rem}
By applying Proposition \ref{computing-left-derived} to the functor $- \underset{\mathcal{D}}{\otimes} -$, we find a practical way to compute homotopy weighted colimits:
\begin{equation}\label{hwcolimits}
 W \underset{\mathcal{D}}{\otimes^h} F =  \gamma_{\mathcal{C}} (W_c \underset{\mathcal{D}}{\otimes} F_c)
\end{equation} 
In fact, as explained in \cite{vokvrinek}, to compute $ W \underset{\mathcal{D}}{\otimes^h} F$, in Equation \ref{hwcolimits} it would be enough to consider the pointwise cofibrant replacement of the diagram $F$ insted of $F_c$.
\end{Rem}

\begin{Def}\label{def-preserve-hcolim}
A homotopical functor $G$ between two $\mathcal{V}$-model categories $\mathcal{C}$ and $\mathcal{E}$ preserves homotopy weighted colimits if, for any pair $(W,F) \in \mathcal{V}^{\mathcal{D}^{op}} \times \mathcal{C}^{\mathcal{D}}$, the map induced in $Ho(\mathcal{C})$ by the zig zag of maps in $\mathcal{C}$ 
\begin{equation}\label{hcolim-map}
W_c \underset{\mathcal{D}}{\otimes} (GF)_c \xleftarrow{\alpha} W_c \underset{\mathcal{D}}{\otimes} (G(F_c))_c \xrightarrow{\beta} W_c \underset{\mathcal{D}}{\otimes} G(F_c) \xrightarrow{\delta} G(W_c \underset{\mathcal{D}}{\otimes} F_c )
\end{equation} 
is an isomorphism (arrows $\alpha$ and $\gamma$ are induced by the factorization through the cofibrant replacement, while $\delta$ is given by the universal property of weighted colimits), or equivalently, that all the maps in \ref{hcolim-map} are weak equivalences.
\end{Def}

\begin{Def}
We say that a left (resp.\ right) total derived functor preserves homotopy weighted colimits if its left (resp.\ right) approximation (as defined in \cite[Paragraph I.III.15]{Approximations}) does so.
\end{Def}

We end the section giving two more definitions involving homotopy weighted colimits, that will be used in the rest of the paper.

\begin{Def}\label{generates}
A full $\mathcal{V}$-enriched subcategory $\mathcal{C}'$ of a $\mathcal{V}$-model category $\mathcal{C}$ is a closed subcategory if it is closed under weak equivalences and weighted homotopy colimits, meaning that every object weakly equivalent to a homotopy weighted colimit of objects in $\mathcal{C}'$ belongs to $\mathcal{C}'$ too.
\end{Def}
\begin{Def}\label{generates2}
Let $\mathcal{C}_0$ be a full $\mathcal{V}$-enriched subcategory of a $\mathcal{V}$-model category $\mathcal{C}$.
The subcategory of $\mathcal{C}$ generated by $\mathcal{C}_0$ is the smallest closed subcategory of $\mathcal{C}$ containing $\mathcal{C}_0$. We write $\langle\mathcal{C}_0\rangle$ for the category generated by $\mathcal{C}_0$. We say that $\mathcal{C}_0$ generates $\mathcal{C}$ if $\langle \mathcal{C}_0 \rangle = \mathcal{C}$. 
\end{Def}

\section{Results} \label{results}

In this whole section, unless otherwise specified, we will assume that $\mathcal{C}$ is a $\mathcal{V}$-model category.

\begin{prop}\label{adjunction}
Let $\mathcal{C}_0$ be a small, full subcategory of $\mathcal{C}$, and suppose that every object of $\mathcal{C}_0$ is cofibrant. Then, the functor 
\begin{equation*}
Hom(\mathcal{C}_0 , -) : \mathcal{C} \rightarrow \mathcal{V}^{\mathcal{C}_0^{op}}; \ \ 
y \longmapsto \mathcal{C}(- , y)
\end{equation*}
given by the postcomposition of the Yoneda embedding with the inclusion functor of $\mathcal{C}_0$ into $\mathcal{C}$, is the right adjoint of a Quillen adjunction.
\end{prop}
\begin{proof}
Called $J : \mathcal{C}_0 \rightarrow \mathcal{C}$ the inclusion functor of $\mathcal{C}_0$ into $\mathcal{C}$, then the left adjoint to $Hom(\mathcal{C}_0 , -)$ is given by the functor
\begin{equation*}
    - \underset{\mathcal{C}_0}{\otimes} J : \mathcal{V}^{\mathcal{C}_0^{op}}\rightarrow \mathcal{C}
\end{equation*}
which sends a presheaf $F$ to $F \underset{\mathcal{C}_0}{\otimes} J$ (see \cite[Theorem 4.50]{Kelly}).

Moreover, the functor $Hom(\mathcal{C}_0 , -)$ preserves fibrations and acyclic fibrations. Indeed, given a (acyclic) fibration $f: a \rightarrow b$ in $\mathcal{C}$, we can prove that 
\begin{equation*}
Hom(\mathcal{C}_0 , a) \xrightarrow{Hom(\mathcal{C}_0 , f)} Hom(\mathcal{C}_0 , b)    
\end{equation*}
is a (acyclic) fibration too. Recall that this last map is a (acyclic) fibration in $\mathcal{V}^{\mathcal{C}_0^{op}}$ if and only if it is objectwise a (acyclic) fibration in $\mathcal{V}$; hence our problem reduces to show that, for any object $d\in\mathcal{C}_0$, the map \begin{equation}\label{fibr}
f\cdot - : \mathcal{C}(d, a) \rightarrow \mathcal{C} (d, b)    
\end{equation}
is a (acyclic) fibration.

By hypothesis, the map $g : \emptyset \rightarrow d$ is a cofibration. 
Moreover, the diagram
\begin{equation*}
\xymatrix{\mathcal{C}(d, b)\ar[d] \ar[r]^{id} & \mathcal{C}(d , b) \ar[d]^{- \cdot g} \\
        \mathcal{C}(\emptyset , a) \ar[r]_{f\cdot -} & \mathcal{C}(\emptyset , b) }
\end{equation*}
is a pullback diagram, which implies, by Definition \ref{SP}, that the map \ref{fibr} is a (acyclic) fibration.
\end{proof}

From a Quillen adjunction between model categories one can always derive an adjunction between the corresponding homotopy categories, by taking the left (resp.\ right) total derived functor of the left (resp.\ right) adjoint (see \cite[Lemma 1.3.10]{Hovey}). This new adjunction is called derived adjunction. In the specific case of Proposition \ref{adjunction}, the derived adjunction will be:
\begin{equation}\label{derived-adjunction}
\mathbf{L}( - \underset{\mathcal{C}_0}{\otimes} J) : Ho(\mathcal{V}^{\mathcal{C}_0^{op}}) \rightleftarrows Ho(\mathcal{C}) : \mathbf{R}Hom(\mathcal{C}_0 , -)
\end{equation}
The left adjoint functor of \ref{derived-adjunction} has the following property:

\begin{prop} \label{L(- x J)-preserves-hcolim}
Let $\mathbf{L}( - \underset{\mathcal{C}_0}{\otimes} J)$ be as in \ref{derived-adjunction}. For any small category $\mathcal{D}$ and for any pair of functors $W : \mathcal{D}^{op} \rightarrow \mathcal{V}$ and $F : \mathcal{D} \rightarrow \mathcal{V}^{\mathcal{C}_0^{op}}$,
\begin{equation*}
\mathbf{L} (- \underset{\mathcal{C}_0}{\otimes} J) (W \underset{\mathcal{D}}{\otimes^h} F ) \xrightarrow{iso} W \underset{\mathcal{D}}{\otimes^h} (\mathbf{L} (- \underset{\mathcal{C}_0}{\otimes} J) (F))    
\end{equation*}
i.e.\ $\mathbf{L}(- \underset{\mathcal{C_0}}{\otimes} J)$ preserves homotopy weighted colimits in the entrance corresponding to the weight.
\end{prop}
\begin{proof}
It is enough to show that the objects $(W_c \underset{\mathcal{D}}{\otimes} F_c) \underset{\mathcal{C}_0}{\otimes J_c}$ and $W_c \underset{\mathcal{D}}{\otimes} (F_c \underset{\mathcal{C_0}}{\otimes} J_c)$ are isomorphic (and therefore weakly equivalent). This is given by the Fubini Theorem for coends (see \cite[Theorem 1.3.1]{Loregian}).
\end{proof}

We now define homotopy tiny objects, who will play a key role in the rest of the section.

\begin{Def}\label{homotopy-tiny}
An object $x$ of $\mathcal{C}$ is said to be homotopy tiny if the total right derived functor $\mathbf{R}\mathcal{C}(x, -) : Ho(\mathcal{C}) \rightarrow Ho(\mathcal{V})$ of the functor 
\begin{equation*}
\mathcal{C}(x, -) : \mathcal{C} \rightarrow \mathcal{V} ; \ \
y \mapsto \mathcal{C} (x, y)
\end{equation*}
preserves homotopy weighted colimits.
\end{Def}
\begin{Rem}
If $y,z$ are fibrant objects of $\mathcal{C}$, $x$ is a cofibrant object of $\mathcal{C}$ and $y \xrightarrow{e} z$ is a weak equivalence in the model structure of $\mathcal{C}$, then the map $\mathcal{C}(x, y) \xrightarrow{e\cdot -} \mathcal{C}(x, z)$ is a weak equivalence. This ensures that the right approximation of $\mathbf{R}\mathcal{C} (x, -)$ is a homotopical functor, and thus it allows us to give Definition \ref{homotopy-tiny}.
\end{Rem}
\begin{Rem}\label{compute-hwcolim-preservation}
Since the approximation of $\mathbf{R} \mathcal{C}(x , -)$ is the functor $\mathcal{C}(x_c , (-)_f) : \mathcal{C} \rightarrow \mathcal{V}$, ``$\mathbf{R} \mathcal{C}(x , -)$ preserves homotopy weighted colimits'' means that, given a small category $\mathcal{D}$ and functors $F: \mathcal{D} \rightarrow \mathcal{C}$, $W: \mathcal{D}^{op} \rightarrow \mathcal{V}$, every map in the zig zag
\begin{equation}\label{htiny-in-model-ct}
    \begin{tikzcd}[column sep={9em,between origins}, row sep={3em,between origins}]
  W_c \underset{\mathcal{D}}{\otimes} (\mathcal{C}(x_c , (-)_f) \cdot F)_c & &  W_c \underset{\mathcal{D}}{\otimes} (\mathcal{C}(x_c , (-)_f) \cdot F_c ) \ar[dr] & \\
  & W_c \underset{\mathcal{D}}{\otimes} (\mathcal{C}(x_c , (-)_f) \cdot F_c)_c \ar[ul]\ar[ur] & & \mathcal{C} (x_c , (W_c \underset{\mathcal{D}}{\otimes} F_c)_f)
\end{tikzcd}
\end{equation}
is a weak equivalence.
\end{Rem}

If we assume that every object of $\mathcal{C}_0$ is homotopy tiny, then adjunction \ref{derived-adjunction} has further interesting properties, as the next two results show.

\begin{prop}\label{Hom(C_0, -)-preserves-hcolim}
Let $\mathcal{C}_0$ be a full subcategory of $\mathcal{C}$, and suppose that every object of $\mathcal{C}_0$ is homotopy tiny. Consider the functor $Hom(\mathcal{C}_0, -)$ defined as in Proposition \ref{adjunction}. Then $\mathbf{R} Hom(\mathcal{C}_0 , -)$ preserves homotopy weighted colimits.
\end{prop}
\begin{proof}
The approximation of $\mathbf{R} Hom(\mathcal{C}_0 , -)$ is the functor $\mathcal{C}((-)_c , (-)_f) : \mathcal{C} \rightarrow \mathcal{V}^{\mathcal{C}_0^{op}}$. Hence, we need to show that every map of the zig zag
\begin{equation*}
    \begin{tikzcd}[column sep={9em,between origins}, row sep={3em,between origins}]
    W_c \underset{\mathcal{D}}{\otimes} (\mathcal{C}((-)_c , (-)_f) \cdot F)_c & &  W_c \underset{\mathcal{D}}{\otimes} (\mathcal{C}((-)_c , (-)_f) \cdot F_c ) \ar[dr] & \\
  & W_c \underset{\mathcal{D}}{\otimes} (\mathcal{C}((-)_c , (-)_f) \cdot F_c)_c \ar[ul]\ar[ur] & & \mathcal{C} ((-)_c , (W_c \underset{\mathcal{D}}{\otimes} F_c)_f)
\end{tikzcd}
\end{equation*}
is a weak equivalence.
Since, in the projective model structure, weak equivalences are the objectwise ones, it is equivalent to show that, for any $x\in \mathcal{C}_0$, every map in the zig zag \ref{htiny-in-model-ct} is a weak equivalence. This follows from the fact that every object of $\mathcal{C}_0$ is homotopy tiny, and in view of Remark \ref{compute-hwcolim-preservation}.
\end{proof}

\begin{Lemma}\label{unit+counit}
Let $\mathcal{C}_0$ be a small, full subcategory of $\mathcal{C}$ whose objects are fibrant, cofibrant and homotopy tiny. Then the adjunction \ref{derived-adjunction} has:
\begin{enumerate}
    \item unit $\eta_{.}$ which is an isomorphism for every representable functor $\mathcal{C}(- , y)$ with $y\in\mathcal{C}_0$;
    \item counit $\epsilon_{.}$ which is an isomorphism for every object $x\in\mathcal{C}_0$.
    \end{enumerate}
\end{Lemma} 
\begin{proof}
\begin{enumerate}
    \item Using, respectively, Lemma \ref{computing-left-derived}, Proposition \ref{repr-functor-cofibrant} and the co-Yoneda Lemma, we get:
\begin{equation*}
    \mathbf{L} (-\underset{\mathcal{C}_0}{\otimes} J) (\gamma_{\mathcal{V}^{\mathcal{C}_0^{op}}} \mathcal{C} ( - , y)) = \gamma_{\mathcal{C}} ((\mathcal{C} ( - , y))_c\underset{\mathcal{C}_0}{\otimes} J) = \gamma_{\mathcal{C}} (\mathcal{C} ( - , y)\underset{\mathcal{C}_0}{\otimes} J) \xrightarrow{iso} \gamma_{\mathcal{C}} y
\end{equation*}

Then,
\begin{equation*}
    \mathbf{R} Hom(\mathcal{C}_0 , -) \mathbf{L} (- \underset{\mathcal{C}_0}{\otimes} J) (\gamma_{\mathcal{V}^{\mathcal{C}_0^{op}}} \mathcal{C} ( - , y))  \xrightarrow{iso} \mathbf{R} Hom(\mathcal{C}_0 , -) (\gamma_{\mathcal{C}} y) = \gamma_{\mathcal{V}^{\mathcal{C}^{op}}} \mathcal{C}((-)_c , y_f)
\end{equation*}

Hence, the map $\eta_{\gamma_{\mathcal{V}^{\mathcal{C}_0^{op}}}\mathcal{V}(- , y)}$ is, up to an isomorphism:
\begin{equation*}
    \eta_{\gamma_{\mathcal{V}^{\mathcal{C}_0^{op}}}\mathcal{V}(- , y)}: \gamma_{\mathcal{V}^{\mathcal{C}^{op}}}\mathcal{C} (- , y) \rightarrow \gamma_{\mathcal{V}^{\mathcal{C}^{op}}} \mathcal{C} ( (-)_c , y_f)
\end{equation*}
Being every object of $\mathcal{C}_0$ fibrant and cofibrant, the two functors $\mathcal{C}(- , y)$ and $\mathcal{C}((-)_c , y_f)$ coincide, and hence $\eta_{\gamma_{\mathcal{V}^{\mathcal{C}_0^{op}}}\mathcal{V}(- , y)}$ is an isomorphism.
\item For any $x\in\mathcal{C}_0$, using the fact that every object in $\mathcal{C}_0$ is both fibrant and cofibrant:
\begin{equation*}
   \mathbf{R} Hom(\mathcal{C}_0 , -) (\gamma_{\mathcal{C}} x) = \gamma_{\mathcal{V}^{\mathcal{C}_0^{op}}} \mathcal{C}((-)_c , x_f) = \gamma_{\mathcal{V}^{\mathcal{C}_0^{op}}} \mathcal{C}(- , x)
\end{equation*}

Then, applying $\mathbf{L} (- \underset{\mathcal{C}_0}{\otimes} J)$ to $\gamma_{\mathcal{V}^{\mathcal{C}_0^{op}}} \mathcal{C}( - , x)$, and using Proposition \ref{repr-functor-cofibrant}, we obtain:
\begin{equation*}
    \mathbf{L} (- \underset{\mathcal{C}_0}{\otimes} J) (\gamma_{\mathcal{V}^{\mathcal{C}_0^{op}}} \mathcal{C}( - , x)) = \gamma_{\mathcal{C}} (- \underset{\mathcal{C}_0}{\otimes} J ) (\mathcal{C} (- , x))_c = \gamma_{\mathcal{C}} (- \underset{\mathcal{C}_0}{\otimes} J) (\mathcal{C} (- , x))
\end{equation*}
Using the co-Yoneda Lemma:
\begin{equation*}
    \gamma_{\mathcal{C}} (- \underset{\mathcal{C}_0}{\otimes}  J ) (\mathcal{C} (- , x)) = \gamma_{\mathcal{C}} ( \mathcal{C} (- , x) \underset{\mathcal{C}_0}{\otimes} J) \xrightarrow{iso} \gamma_{C} x
\end{equation*}
which concludes the proof that $\epsilon_{\gamma_{\mathcal{C}} x}$ is an isomorphism. \qedhere
\end{enumerate}
\end{proof}

\begin{Rem}
The functor $ - \underset{\mathcal{C}_0}{\otimes} J$ has value in the subcategory $\langle \mathcal{C}_0 \rangle$, since every $F \underset{\mathcal{C}_0}{\otimes} J$ is a weighted colimit of a diagram of objects in $\mathcal{C}_0$. This allows us to restrict the adjunction in \ref{derived-adjunction} to an adjunction between $Ho(\mathcal{V}^{\mathcal{C}_0^{op}})$ and $Ho(\langle \mathcal{C}_0 \rangle)$.
\end{Rem}

\begin{Th}\label{Th}
Let $\mathcal{C}_0$ be a small, full subcategory of $\mathcal{C}$ of fibrant and cofibrant homotopy tiny objects. Then adjunction \ref{derived-adjunction} induces an equivalence between $Ho(\langle\mathcal{C}_0 \rangle)$ and $Ho(\mathcal{V}^{\mathcal{C}_0^{op}})$.
\end{Th}
\begin{proof}
We prove the thesis by showing that the unit and the counit are isomorphisms.

\begin{enumerate}
    \item Every $F\in\mathcal{V}^{\mathcal{C}_0^{op}}$ is isomorphic to $F \underset{\mathcal{C}_0}{\otimes} \mathcal{Y}$, where $\mathcal{Y} : \mathcal{C}_0 \rightarrow \mathcal{V}^{\mathcal{C}_0^{op}}$ denotes the Yoneda embedding. Then, by applying, in order, Proposition \ref{L(- x J)-preserves-hcolim},  Proposition \ref{Hom(C_0, -)-preserves-hcolim}. and Point (1) of Lemma \ref{unit+counit}, we conclude that $\eta_{.}$ is an isomorphism for every object of $\mathcal{V}^{\mathcal{C}_0^{op}}$.
    \item For any small category $\mathcal{D}$, and for any pair of functors $F : \mathcal{D} \rightarrow \mathcal{C}_0$, $W: \mathcal{D}^{op}\rightarrow \mathcal{V}$, if $x\in \mathcal{C}$ is weakly equivalent to $W_c \underset{\mathcal{D}}{\otimes} F_c$, then $\epsilon_{\gamma_{\mathcal{C}} x}$ is an isomorphism. Indeed:
\begin{itemize}
\item[$\diamond$] if $y \sim z$ in the model category $\mathcal{C}$, and if $\epsilon_{\gamma_{\mathcal{C}} y}$ is an isomorphism, then $\epsilon_{\gamma_{\mathcal{C}} z}$ is an isomorphism too.
This comes easily from the fact that $\gamma_{\mathcal{C}}(y)$ and $\gamma_{\mathcal{C}} (z)$ are isomorphic, and by looking at the commutative square:
\begin{equation*}
\xymatrix{ \mathbf{L} (- \underset{\mathcal{C}_0}{\otimes} J) \mathbf{R} Hom(\mathcal{C}_0, -) \gamma_{\mathcal{C}} (y) \ar[d] \ar[r]^-{\epsilon_{\gamma_{\mathcal{C}} y}} & \gamma_{\mathcal{C}} (y) \ar[d] \\
\mathbf{L} (- \underset{\mathcal{C}_0}{\otimes} J) \mathbf{R} Hom(\mathcal{C}_0, -) \gamma_{\mathcal{C}} (z) \ar[r]_-{\epsilon_{\gamma_{\mathcal{C}} z}} & \gamma_{\mathcal{C}} (z)}    
\end{equation*}
    \item[$\diamond$] given two functors $F: \mathcal{D} \rightarrow \mathcal{C}_0$ and $W : \mathcal{D}^{op} \rightarrow \mathcal{V}$, then $\epsilon_{W \underset{\mathcal{D}}{\otimes} F}$ is an isomorphism: this follows from Proposition \ref{Hom(C_0, -)-preserves-hcolim}, Proposition \ref{L(- x J)-preserves-hcolim}, and Point (2) of Lemma \ref{unit+counit}. \qedhere
\end{itemize}
\end{enumerate}
\end{proof}

\begin{Rem}
If $\mathcal{C}_0$ is a subcategory of $\mathcal{C}$ whose objects are homotopy tiny, but not fibrant and cofibrant, one could obtain a new category $\mathcal{C}'_0$ by taking fibrant and cofibrant replacements of every object of $\mathcal{C}_0$. It is easy to verify that $\langle \mathcal{C}_0 \rangle$ and $\langle \mathcal{C}'_0 \rangle$ coincide. In this case, Theorem \ref{Th} gives us the existence of an equivalence between $Ho(\langle \mathcal{C}_0 \rangle)$ and $Ho(\mathcal{V}^{{\mathcal{C}'_0}^{op}})$.
\end{Rem}

\begin{Rem}
The equivalence of Theorem \ref{Th} is not, in general, a Quillen equivalence, since the category $\langle\mathcal{C}_0\rangle$ is not even, in general, a model category itself. However, if it happens that  $\langle\mathcal{C}_0\rangle$ is a model category, then this equivalence turns out to be a Quillen equivalence, as the argument in the proof of Proposition \ref{adjunction} can be repeated with $\mathcal{C}$ replaced by $\langle \mathcal{C}_0 \rangle$.
\end{Rem}

\begin{Cor}\label{result}
In the hypothesis of Theorem \ref{Th}, suppose that $\mathcal{C}$ is generated by $\mathcal{C}_0$. Then $\mathcal{C}$ and $\mathcal{V}^{\mathcal{C}_0^{op}}$ are Quillen equivalent.
\end{Cor}

\section{Motivating examples} \label{examples}
In this last section, we see how our result includes both Elmendorf's Theorem and Schwede--Shipley's Theorem. We show this by state these theorems and prove them as consequences of Corollary \ref{result}.

\subsection{Elmendorf's Theorem}
Let $\mathcal{V}$ be the category of compactly generated spaces, and let $\mathcal{O}_G$ be the set of orbits of a compact Lie group $G$. We denote by $\mathcal{V}^G$ the category of $G$-spaces.
\begin{Th}[Elmendorf's Theorem] \label{Elm-th}
There exists a Quillen equivalence 
\begin{equation*}
    \theta : \mathcal{V}^{\mathcal{O}_{G}^{op}} \rightleftarrows \mathcal{V}^{G} : \phi
\end{equation*}
where $\mathcal{V}^{\mathcal{O}_{G}^{op}}$ is equipped with the projective model structure, whereas $\mathcal{V}^{G}$ is equipped with the $G$-fixed point model structure.
\end{Th}
More details about Elmendorf's Theorem can be found in \cite{Elmendorf}, where one can find that the functor $\phi$ is defined (up to an isomorphism) as the functor sending $x$ to $\mathcal{V}^{\mathcal{O}^{op}_{G}}(- , x)$, as in the case of the right adjoint of Proposition \ref{adjunction}. Moreover:
\begin{Lemma}\label{lemma-Elm}
In the category $\mathcal{V}^G$,  orbits are homotopy tiny objects.
\end{Lemma}
\begin{proof}
Fix $H$ closed subgroup of $G$, and consider two functors $W : \mathcal{D}^{op}\rightarrow \mathcal{V}$ and $F: \mathcal{D} \rightarrow \mathcal{V}^G$. We want to show that
    \begin{equation}\label{Elm}
    \mathbf{R}\mathcal{V}^G \left( \frac{G}{H}, W \overset{h}{\underset{\mathcal{D}}{\otimes}} F \right) \xrightarrow{iso} W \overset{h}{\underset{\mathcal{D}}{\otimes}} \mathbf{R} \mathcal{V}^G \left( \frac{G}{H}, F \right)
    \end{equation}
    WLOG, we assume that $W$ is cofibrant, and that $F$ is fibrant and cofibrant (we don't need any further assumption of $\frac{G}{H}$, since an orbit is always cofibrant). Then, to show \ref{Elm} is the same as to show that every map of the composition
     \begin{equation}\label{comp}
            W \underset{\mathcal{D}}{\otimes} \left( \mathcal{V}^G \left( \frac{G}{H}, F \right) \right)_c \rightarrow W \underset{\mathcal{D}}{\otimes} \mathcal{V}^G \left( \frac{G}{H}, F \right) \rightarrow \mathcal{V}^G \left( \frac{G}{H}, (W \underset{\mathcal{D}}{\otimes} F)_f \right)
    \end{equation}
  is a weak equivalence.

    By \cite[Lemma 3.8]{Elmendorf}, composition in \ref{comp} is the same, up to isomorphism, as 
    \begin{equation*}
    W \underset{\mathcal{D}}{\otimes} ((F)^H)_c \rightarrow W \underset{\mathcal{D}}{\otimes} (F)^H \rightarrow ((W \underset{\mathcal{D}}{\otimes} F)_f)^H
    \end{equation*}
    
    The functor $(-)^H$ preserves cofibrant objects; hence, $((F)^H)_c = F^H$, and the thesis reduces to show that the map $W \underset{\mathcal{D}}{\otimes} (F)^H \rightarrow ((W \underset{\mathcal{D}}{\otimes} F)_f)^H$ is a weak equivalence.
    
   By looking at the characterization of $W \underset{\mathcal{D}}{\otimes} F$ as a coequalizer (see \cite[Remark 1.2.4]{Loregian}), we have
   \begin{equation*}
       (W \underset{\mathcal{D}}{\otimes} F) = coeq ( \underset{d \rightarrow d'}{\coprod} (Wd' \otimes Fd) \rightrightarrows  \underset{d \in \mathcal{D}}{\coprod} (Wd \otimes Fd) ) \end{equation*}
and therefore 
   \begin{equation*}
       (W \underset{\mathcal{D}}{\otimes} F)^H = coeq ( (\underset{d \rightarrow d'}{\coprod} (Wd' \otimes Fd))^H \rightrightarrows  (\underset{d \in \mathcal{D}}{\coprod} (Wd \otimes Fd))^H) 
       \end{equation*}
       
 Since $G$ a compact Lie group, fixed point functors commute with Bousfield--Kan homotopy colimits (see \cite[Proposition 1.2]{Malkiewich}), and therefore with homotopy colimits (see \cite[Corollary 5.2]{Bousfield}); hence, fixed point functors commute with coproducts, and \ref{step 2} can be rewritten, up to an isomorphism, as:
    \begin{equation}\label{step 2}
       (W \underset{\mathcal{D}}{\otimes} F)^H = coeq (\underset{d \rightarrow d'}{\coprod} (Wd' \otimes Fd)^H \rightrightarrows  \underset{d \in \mathcal{D}}{\coprod} (Wd \otimes Fd)^H ) \end{equation}
   Then, we note that any object $Wd \in \mathcal{V}$ can be seen as a $G$-space where $G$ acts trivially. This means that, for any given closed subgroup $H$ of $G$, 
   \begin{equation} \label{tool}
    (Wd \otimes Fd)^H = \underset{H}{lim} (Wd \otimes Fd) \xrightarrow{iso} Wd \otimes (\underset{H}{lim} Fd) = Wd \otimes (Fd)^H
   \end{equation}
  Using \ref{tool}, from \ref{step 2} we obtain:    \begin{equation*}
       (W \underset{\mathcal{D}}{\otimes} F)^H \xrightarrow{iso} coeq (\underset{d \rightarrow d'}{\coprod} (Wd' \otimes (Fd)^H) \rightrightarrows  \underset{d \in \mathcal{D}}{\coprod} (Wd \otimes (Fd) )^H) \overset{def}{=} W \underset{\mathcal{D}}{\otimes} F^H \end{equation*}
 Since $(-)^H$ preserves fibrant replacements, then the map $(W \otimes F)^H \rightarrow (W \otimes F)_f^H$ is a weak equivalence, and this concludes the proof.
\end{proof}
We can now present a proof of Elmendorf's Theorem based on Corollary \ref{result}:
\begin{proof}[Proof of Theorem \ref{Elm-th}]
Every $x$ in $\mathcal{V}^G$ is weakly equivalent to a $G$-CW-complex (see \cite[Theorem 3.6]{May}), and every $G$-CW-complex has $n$-skeleton given by the pushout of weighted colimits of orbits. Hence, $\mathcal{V}^G = \langle \mathcal{O}_G \rangle $. Moreover, every orbit is fibrant and cofibrant, and, by Lemma \ref{lemma-Elm}, is a homotopy tiny object of the category $\mathcal{V}^G$ equipped with the $G$-fixed model structure. Hence, by Corollary \ref{result}, $\mathcal{V}^G$ and $\mathcal{V}^{\mathcal{O}^{op}_G}$ are Quillen equivalent.
\end{proof}

\subsection{The Schwede--Shipley Theorem}
We first give an overview of the description of stable model categories in terms of presehaves categories presented in \cite{Schwede-Shipley}. Here, after the definition of the category $\mathcal{S}p (\mathcal{C})$ (\cite[Definition 3.6.1]{Schwede-Shipley}, we find:
\begin{Th} \label{SS-preliminar-th1} (\cite[Theorem 3.8.2]{Schwede-Shipley}) Let $\mathcal{C}$ be a simplicial, cofibrantly generated, proper, stable model category. Then $\mathcal{S}p (\mathcal{C})$ supports the stable model
structure and the suspension spectrum functor $\Sigma_{\infty}$ and the evaluation functor $Ev_0$ form a Quillen equivalence between $\mathcal{C}$ and $\mathcal{S}p (\mathcal{C})$ with the stable model structure.
\end{Th}
\begin{proof}
See \cite[Section 3.8]{Schwede-Shipley}, where the reader can also find the description of the stable model structure.
\end{proof}
Generators are defined using localizing subcategories (cfr \cite[Definition 2.1.2]{Schwede-Shipley}):
\begin{Def}\label{localizing}
A localizing subcategory of $Ho(\mathcal{C})$ is defined as a full triangulated subcategory of $Ho(\mathcal{C})$ which is closed under (homotopy) coproducts.
\end{Def}
\begin{Rem}
Working within the model category $\mathcal{C}$, a localizing subcategory will therefore correspond to a full subcategory closed under weak equivalences (since a triangulated subcategory is closed under isomorphisms) and coproducts.
\end{Rem}
\begin{Def} \label{SS-generatated by}
The localizing subcategory generated by a full subcategory $\mathcal{G}$ of $\mathcal{C}$ is the smallest localizing subcategory of $Ho(\mathcal{C})$ containing $\gamma_{\mathcal{C}}(\mathcal{G})$. We denote by $\langle \mathcal{G} \rangle _L$ the subcategory of $\mathcal{C}$ whose image via $\gamma_{\mathcal{C}}$ is the localizing subcategory generated by $\mathcal{G}$.
\end{Def}
\begin{Def}\label{SS-generators}
$\mathcal{G}$ is a subcategory of generators for $\mathcal{C}$ if $\langle \mathcal{G} \rangle _L = \mathcal{C}$.
\end{Def}
\begin{Rem}
Definition \ref{localizing} is not the same as Definition \ref{generates}. Anyway, if we consider a subcategory $\mathcal{G}$ of generators of $\mathcal{C}$ in the sense of Definition \ref{SS-generators}, then $\langle \mathcal{G} \rangle \supseteq \langle \mathcal{G} \rangle _L = \mathcal{C}$. Hence, the subcategory $\mathcal{G}$ generates $\mathcal{C}$ also in the sense of Definition \ref{generates2}.
\end{Rem}
Compact objects of a stable model category $\mathcal{C}$ are defined in the following way (cfr \cite[Definition 2.1.2]{Schwede-Shipley}):
\begin{Def}
Let $\mathcal{T}$ be a triangulated category with small coproducts. An object $x$ of $\mathcal{T}$ is compact if, for any small diagram $F : \mathcal{D} \rightarrow \mathcal{C}$, the universal map
\begin{equation*}
    \underset{\mathcal{D}}{\coprod} \mathcal{T}( x , F) \rightarrow \mathcal{T} (x, \underset{\mathcal{D}}{\coprod} F )
\end{equation*}
is an isomorphism. An object $x$ of a stable model category is called compact if $\gamma_{\mathcal{C}}x$ is compact in the triangulated homotopy category.
\end{Def}

\cite[Theorem 3.9.3.(iii)]{Schwede-Shipley} asserts:
\begin{Th}\label{SS-preliminar-th2}
Let $\mathcal{C}$ be a spectral model category,  $\mathcal{G}$ the full subcategory of compact, fibrant and cofibrant generators for $\mathcal{C}$, $Sp^{\Sigma}$ the category of symmetric spectra, and let
$ - \underset{\mathcal{G}}{\wedge} J $ and $Hom(\mathcal{G}, -)$ defined as the adjoint functors of Proposition \ref{adjunction}. Then the pair of functors
\begin{equation*}
    - \underset{\mathcal{G}}{\wedge} J : (Sp^{\Sigma})^{\mathcal{G}^{op}} \rightleftarrows \mathcal{C} : Hom(\mathcal{G} , -)
\end{equation*}
forms a Quillen equivalence.
\end{Th}
\begin{proof}
See \cite[Theorem 3.9.3]{Schwede-Shipley}.
\end{proof}
Using Theorem \ref{SS-preliminar-th1} and Theorem \ref{SS-preliminar-th2}, one could derive:
\begin{Th} [Schwede--Shipley's Theorem] \cite[Theorem 3.3.3]{Schwede-Shipley} \label{SS-th}
Let $\mathcal{C}$ be a simplicial, cofibrantly generated, proper, stable model category with a set $\mathcal{G}$ of fibrant, cofibrant, compact generators. Then there
exists a chain of simplicial Quillen equivalences between $\mathcal{C}$ and the model category $Sp(\Sigma)^{(\mathcal{G})^{op}}$.
\end{Th}

In the following, we will give a new proof, based on Corollary \ref{result}, of Theorem \ref{SS-th}, that we rephrase as:
\begin{Th}\label{SS-rephrased}
Let $\mathcal{C}$ be a simplicial, cofibrantly generated, proper, stable model category with a set $\mathcal{G}$ of compact, fibrant and cofibrant generators. Then $\mathcal{C}$ is Quillen equivalent to the category $Sp(\Sigma)^{(\mathcal{G})^{op}}$.
\end{Th}

To prove Theorem \ref{SS-rephrased}, we first show that, in a stable model category $\mathcal{C}$, any cofibrant object $x$ is compact if and only if it is homotopy tiny.

\begin{prop}\label{Prop1-SS}
Let $\mathcal{C}$ be a stable model category. Then every homotopy tiny object is compact.
\end{prop}
\begin{proof}
In a stable model category $\mathcal{C}$, homotopy coproducts are image of coproducts under $\gamma_{\mathcal{C}}$; hence, given a homotopy tiny object $x$, since coproducts are special case of weighted colimits,
\begin{equation*}
    Ho(\mathcal{C})(\gamma_{\mathcal{C}} x , \coprod \gamma_{\mathcal{C}^{\mathcal{D}}} D) \xrightarrow{iso} \coprod Ho(\mathcal{C}) (\gamma_{\mathcal{C}} x , \gamma_{\mathcal{C}^{\mathcal{D}}}D)
\end{equation*}
which implies that $x$ is, by definition, a compact object.
\end{proof}

\begin{prop}\label{Prop2-SS}
Let $\mathcal{C}$ be a stable model category, and let $x$ be a cofibrant compact object. Then $x$ is homotopy tiny.
\end{prop}
The next definition and the following two lemmas are preliminary to the proof of Proposition \ref{Prop2-SS}.
\begin{Def}\label{goodness}
An object $\sigma$ of $Sp^{\Sigma}$ is good if, for any $x$ cofibrant object of $\mathcal{C}$ and for any $y\in\mathcal{C}$, the composition
\begin{equation*}
        \sigma_c \otimes \mathcal{C}(x , y_f)_c \rightarrow \sigma_c \otimes \mathcal{C} (x, y_f) \rightarrow \sigma_c \otimes \mathcal{C}(x , (y_c)_f) \rightarrow \mathcal{C} (x , (\sigma_c \wedge y_c)_f) 
\end{equation*}
is a weak equivalence.
\end{Def}
\begin{Lemma}\label{Lemma1-SS}
The collection of all good objects coincides with $Ob(Sp^{\Sigma})$.
\end{Lemma}
\begin{proof}
We divide the proof into five points:
\begin{itemize}
    \item[$\diamond$] If $\sigma_1 \sim \sigma_2$ and $\sigma_1$ is good, then $\sigma_2$ is good: WLOG, we assume that $\sigma_1$ and $\sigma_2$ are cofibrant, and we look at the commutative square
    \begin{equation*}
        \xymatrix{
        \sigma_1 \otimes \mathcal{C}(x , y_f)_c \ar[r]\ar[d]_{\alpha} & \sigma_1 \otimes \mathcal{C} (x, y_f) \ar[r] & \sigma_1 \otimes \mathcal{C}(x, (y_c)_f) \ar[r] & \mathcal{C} (x , (\sigma_1 \wedge y_c)_f) \ar[d]^{\beta} \\
        \sigma_2 \otimes \mathcal{C}(x , y_f)_c \ar[r] & \sigma_2 \otimes \mathcal{C} (x, y_f) \ar[r] &  \sigma_2 \otimes \mathcal{C}(x, (y_c)_f) \ar[r] & \mathcal{C} (x , (\sigma_2 \wedge y_c)_f) \\
        }
    \end{equation*}
    where we assume that the top composition of arrows is a weak equivalence.
    $\alpha$ and $\beta$ can be shown to be weak equivalences too: indeed, for any given acyclic cofibration from $\sigma_1$ to $\sigma_2$, the pushout-product axiom gives a weak equivalence between $\sigma_1 \otimes \mathcal{C}(x,y_f)_c$ and $\sigma_2 \otimes \mathcal{C}(x,y_f)_c$ (resp.\ between $\sigma_1 \wedge y_c$ and $\sigma_2 \wedge y_c$); using Ken Brown's Lemma, this result can be extended to any weak equivalence (and so to any zig zag of weak equivalences) between $\sigma_1$ and $\sigma_2$. 
    \item[$\diamond$] $S^1$ is good: using the definition of stable model category as it is given in \cite[Definition 2.1.1]{Schwede-Shipley}, the suspension functor $\Sigma$ and the loop functor $\Omega$ are inverse equivalences in the homotopy category. Then,
    \begin{equation}\label{iso-in-HoC}
    \begin{split}
        \Sigma ( Ho(\mathcal{C})(\gamma_{\mathcal{C}} x, \gamma_{\mathcal{C}} y)) & \xrightarrow{iso}  \Sigma ( Ho(\mathcal{C})(\Sigma \gamma_{\mathcal{C}} x , \Sigma \gamma_{\mathcal{C}} y)) \\ & \xrightarrow{iso} \Sigma ( Ho(\mathcal{C})(\gamma_{\mathcal{C}} x, \Omega \Sigma \gamma_{\mathcal{C}} y)) \\ & \xrightarrow{iso} \Sigma (\Omega Ho(\mathcal{C}) (\gamma_{\mathcal{C}} x, \Sigma \gamma_{\mathcal{C}} y)) \\
        & \xrightarrow{iso} Ho(\mathcal{C})(\gamma_{\mathcal{C}} x, \Sigma \gamma_{\mathcal{C}} y))
        \end{split}
    \end{equation}
    Translating the isomorphism given by the composition in \ref{iso-in-HoC} in the model category $\mathcal{C}$, and using that $x$ and $S^1$ are cofibrant, we obtain a weak equivalence:
        \begin{equation*}
        S^1 \otimes \mathcal{C}(x,y_f)_c \xrightarrow{\sim} \mathcal{C}(x, (S^1 \wedge y_c)_f)
    \end{equation*}
    \item[$\diamond$] $S^{-1}$ is good: we want to prove that the composition
    \begin{equation}\label{zig zag S^-1}
        S^{-1} \otimes \mathcal{C}(x, y_f)_c \rightarrow S^{-1} \otimes \mathcal{C}(x, y_f) \rightarrow S^{-1} \otimes \mathcal{C}(x, (y_c)_f) \rightarrow \mathcal{C}(x, (S^{-1}\wedge y_c)_f)
    \end{equation}
    is a weak equivalence.

    Being the suspension functor an equivalence on the homotopy category, the composition in \ref{zig zag S^-1} is a weak equivalence if and only if
\begin{equation}\label{zig zag S^-1 with S^1}
\begin{tikzcd}[column sep={7em,between origins}, row sep={3em,between origins}]
 S^1 \otimes S^{-1} \otimes \mathcal{C}(x, y_f)_c \ar[dr] & &  S^1 \otimes (S^{-1} \otimes \mathcal{C}(x, (y_c)_f) )_c \ar[dr] & \\
 & S^1 \otimes (S^{-1} \otimes \mathcal{C}(x, y_f))_c \ar[ur] & & S^1 \otimes \mathcal{C}(x, (S^{-1}\wedge y_c)_f)_c
\end{tikzcd}    
\end{equation}

    is a weak equivalence.
    The composition in \ref{zig zag S^-1 with S^1} is, up to an isomorphism,
\begin{equation}\label{isomorphic zig zag}
\begin{tikzcd}[column sep={8em,between origins}, row sep={3em,between origins}]
 \mathcal{C}(x, y_f)_c \ar[dr] & &  S^1 \otimes (S^{-1} \otimes \mathcal{C}(x, (y_c)_f) )_c \ar[dr] & \\
 & S^1 \otimes (S^{-1} \otimes \mathcal{C}(x, y_f))_c \ar[ur] & & S^1 \otimes \mathcal{C}(x, (S^{-1}\wedge y_c)_f)_c
\end{tikzcd}        
\end{equation}
and we can prove that \ref{isomorphic zig zag} is a weak equivalence by considering the following commutative square:
    \begin{equation*}
        \xymatrix{ \mathcal{C} (x, y_f)_c \ar[r]\ar[d]_{\sim} & S^1 \otimes (S^{-1} \otimes \mathcal{C} (x, y_f))_c \ar[r] & S^1 \otimes (S^{-1} \otimes \mathcal{C}(x, (y_c)_f))_c \ar[d] \\
        \mathcal{C}(x, (y_c)_f)_c \ar[r]^-{=} & \mathcal{C}(x, (S^1 \wedge (S^{-1} \wedge y_c))_f) & S^1 \otimes \mathcal{C} (x, (S^{-1} \wedge y_c)_f)_c \ar[l]_{\sim}^{S^1 good} 
         }
    \end{equation*}
    \item[$\diamond$] $S^0$ is good: this follows from the isomorphisms, for any $\sigma \in Sp^\Sigma$ and $y\in\mathcal{C}$, between $S^0 \otimes \sigma$ and $\sigma$, and $S^0 \wedge y$ and $y$.
    \item[$\diamond$] if $\sigma_1$ and $\sigma_2$ are good, then $\sigma_1 \otimes \sigma_2$ is good too: WLOG, we assume that $\sigma_1$ and $\sigma_2$ are cofibrant, and we look at the commutative diagram
    \begin{tiny}
    \begin{equation*}
        \xymatrix{ & (\sigma_1 \otimes \sigma_2) \otimes \mathcal{C}(x, y_f) \ar[r] & \sigma_1 \otimes \sigma_2 \otimes \mathcal{C}(x, (y_c)_f) \ar[dr] & \\
        (\sigma_1 \otimes \sigma_2) \otimes (\mathcal{C}(x, y_f))_c \ar[ur]\ar[dd]^{\wr}_{\sigma_2 \ good} & & & \mathcal{C}(x, ((\sigma_1 \otimes \sigma_2) \wedge y_c)_f)\ar[dd]_{=} \\
        & & & \\
        \sigma_1 \otimes \mathcal{C}(x, (\sigma_2 \wedge y_c)_f) \ar[dr]^{\sim}_{\sigma_1 good} & & & \mathcal{C}(x, ( \sigma_1 \wedge (\sigma_2 \wedge y_c))_f) \ar[dl]_{\epsilon}^{\sim} \\ & \mathcal{C}(x, (\sigma_1 \wedge ( (\sigma_2 \wedge y_c)_f)_c)_f) & \mathcal{C}(x, ( \sigma_1 \wedge (\sigma_2 \wedge y_c)_f)_f) \ar[l]_-{\delta}^{=} & }
    \end{equation*}
    \end{tiny}
    where $\epsilon$ is a weak equivalence because $(\sigma_2 \wedge y_c) \rightarrow (\sigma_2 \wedge y_c)_f$ is an acyclic cofibration between cofibrant objects, and the functor $\sigma_1 \wedge -$ preserves acyclic cofibrations.
\end{itemize}
 We have shown that the collection of all good objects of $Sp^{\Sigma}$ contains $S^0, S^1, S^{-1}$ and it is closed under weak equivalences and tensor product: this is precisely $Ob(Sp^{\Sigma})$.
\end{proof}

\begin{Lemma} \label{Lemma2-SS}
For any $x$ compact object of a stable model category $\in\mathcal{C}$, $\mathbf{R}\mathcal{C}(x, -)$ preserves homotopy colimits.
\end{Lemma}
\begin{proof}
It is well known that to any model category $\mathcal{C}$ it is possible to associate an infinity category $F\mathcal{C}$ s.t., given a functor $G$ between two model categories $\mathcal{C}$ and $\mathcal{D}$, the corresponding functor $FG : F\mathcal{C} \rightarrow F\mathcal{D}$ preserves colimits if and only if $G$ preserves homotopy colimits; one can then recover the fact that $\mathbf{R} \mathcal{C}(x, -)$ preserves homotopy colimits from the fact that it preserves homotopy pushouts and coproducts, and using \cite[Proposition 4.4.2.7]{Lurie} for $k = \omega$. Indeed $\mathbf{R}\mathcal{C} (x, -)$ preserves homotopy pushout squares, since, in stable model categories, pushout squares and pullback squares coincide; also, being $x$ compact, $\mathbf{R} \mathcal{C}(x, -)$ preserves homotopy coproducts.
\end{proof}
We now have all the ingredients we need in order to prove Proposition \ref{Prop2-SS}.
\begin{proof}[Proof of Proposition \ref{Prop2-SS}]
Let $x$ be a cofibrant compact object of $\mathcal{C}$. We want to show that 
\begin{equation} \label{isom}
    \mathbf{R}\mathcal{C} (x , -) ( W \underset{\mathcal{D}}{\otimes}^h F)) \xrightarrow{iso} W \underset{\mathcal{D}}{\otimes^h}, \mathbf{R}\mathcal{C} (x , -) (F)
\end{equation}
Using the description of homotopy weighted colimits in terms of homotopy colimits on the twisted arrow category of $\mathcal{D}$ (see \cite[Remark 1.2.3]{Loregian}), \ref{isom} becomes: 
    \begin{equation*}
    \mathbf{R}\mathcal{C} (x, -) (\underset{d \rightarrow d'}{hcolim}(\gamma_{Sp^{\Sigma}}W(d')\wedge^h \gamma_{\mathcal{C}} F(d)) \xrightarrow{iso} \underset{d \rightarrow d'}{hcolim} ( \gamma_{Sp^{\Sigma}} W(d') \otimes^h \mathbf{R}\mathcal{C} (x , -) (\gamma_\mathcal{C} (F(d))_f))
    \end{equation*}
    which can be shown using Lemma \ref{Lemma1-SS} and Lemma \ref{Lemma2-SS}.
\end{proof}
\begin{Cor}\label{Corollary-SS}
Let $\mathcal{C}$ be a stable model category, and let $x$ be a cofibrant object. Then $x$ is compact if and only if it is homotopy tiny.
\end{Cor}
\begin{proof}
Straight from Proposition \ref{Prop1-SS} and Proposition \ref{Prop2-SS}.
\end{proof}
We finally prove Theorem \ref{SS-th}:
\begin{proof}[Proof of Theorem \ref{SS-th}]
Straight from Corollary \ref{Corollary-SS} and Corollary \ref{result}.
\end{proof}

As anticipated as the beginning of the section, we have shown that Corollary \ref{result} is a generalization of Elmendorf's Theorem and Scwede--Shipley's Theorem: indeed, these can both be proven as corollaries of our result.

\newpage
\printbibliography

\end{document}